\documentclass[twoside]{article}

\usepackage[margin=2cm]{geometry}
\usepackage{amsmath,amsthm,amssymb,latexsym}
\usepackage[v2,tips,curve]{xy}
\usepackage{dsfont}
\usepackage{graphicx}

\theoremstyle{plain}
\newtheorem{proposition}{Proposition}[section]
\newtheorem{theorem}[proposition]{Theorem}
\newtheorem{corollary}[proposition]{Corollary}
\newtheorem{lemma}[proposition]{Lemma}
\theoremstyle{definition}
\newtheorem{definition}[proposition]{Definition}
\newtheorem{remark}[proposition]{Remark}

\newcommand{\vnten}{{\overline\otimes}}
\newcommand{\mc}{\mathcal}

\newcommand{\G}{\mathbb{G}}
\renewcommand{\H}{\mathbb{H}}
\newcommand{\K}{\mathbb{K}}

\newcommand{\ip}[2]{\langle{#1},{#2}\rangle}

\newcommand{\hh}{\widehat}
\newcommand{\id}{\operatorname{id}}
\newcommand{\Gr}{\operatorname{Gr}}
\newcommand{\spec}{\operatorname{sp}}

\newcommand{\ww}{W}
\newcommand{\WW}{\mathbb{W}}
\newcommand{\Ww}{\mathds{W}}
\newcommand{\wW}{\text{\reflectbox{$\Ww$}}\:\!} 

\begin{document}

\large
\title{Categorical aspects of quantum groups: multipliers and intrinsic groups}
\author{Matthew Daws}
\maketitle

\begin{abstract}
We show that the assignment of the (left) completely bounded multiplier algebra
$M_{cb}^l(L^1(\mathbb G))$ to a locally compact quantum group $\mathbb G$, and
the assignment of the intrinsic group, form functors between appropriate
categories.  Morphisms of locally compact quantum
groups can be described by Hopf $*$-homomorphisms between universal
$C^*$-algebras, by bicharacters, or by special sorts of coactions.
We show that the whole
theory of completely bounded multipliers can be lifted to the universal
$C^*$-algebra level, and that then the different pictures of both multipliers
(reduced, universal, and as centralisers)
and morphisms interact in extremely natural ways.  The intrinsic group of a
quantum group can be realised as a class of multipliers, and so our techniques
immediately apply.  We also show how to think of the intrinsic group using
the universal $C^*$-algebra picture, and then, again, show how the differing
views on the intrinsic group interact naturally with morphisms.  We show that
the intrinsic group is the ``maximal classical'' quantum subgroup of a locally
compact quantum group, show that it is even closed in the strong Vaes sense,
and that the intrinsic group functor is an adjoint to the inclusion functor
from locally compact groups to quantum groups.


\end{abstract}

\section{Introduction}

Locally compact quantum groups, \cite{kv}, are now recognised
as the correct way to think about quantum groups from an operator algebraic
viewpoint; or alternatively, as a non-commutative extension of Pontryagin
Duality, building on the theory of Kac Algebras and Multiplicative Unitary
theory.  Thinking from a category theory perspective, it is also extremely
interesting to think about what the morphisms of such objects should be.
The recent paper \cite{mrw}, building on \cite{ng, kusuni}, gives a very satisfactory
answer.

As already seen for a non-amenable group $G$, when forming the group algebra
we can either look at the reduced case $C^*_r(G)$ or the universal case
$C^*(G)$, which are different.  The same applies to a quantum group $\G$ giving,
in general different, algebras $C_0(\G)$ and $C_0^u(\G)$.  Morphisms can either
be expressed as Hopf $*$-homomorphisms, at the universal algebra level, or
through bicharacters, or special types of coaction, at the reduced algebra level
(or equivalently at the von Neumann algebra level, compare
\cite[Section~12]{kusuni}).
This gives further weight to the view that $C_0(\G)$ and $C_0^u(\G)$ are just
different facets of the same ``quantum group'' $\G$.

There has been much interest recently in studying completely bounded multipliers
of $\G$, or more precisely, of the convolution algebra $L^1(\G)$, see
\cite{daws1, daws2, hnr, jnr} for example.
This is an $L^1(\G)$ theory, or equivalently an $L^\infty(\G)$ theory.
In this paper we show how to cast the entire theory using $C_0^u(\G)$,
again showing that the distinction between $C_0(\G)$ and $C_0^u(\G)$ is in
many ways unimportant.  We believe that, even for completely bounded multipliers
of the Fourier algebra, this approach, that is, using $C^*(G)$ and not $VN(G)$,
has not previously been studied.

Using this theory, in terms of universal algebras, leads to a simple way to
study how multipliers and morphisms interact, generalising the study Spronk
undertook of cb multipliers of $A(G)$ in \cite[Section~6.1]{spronk}.
We then find that all the different ``pictures'' of morphisms have a very natural
interpretation with the different aspects of multipliers (essentially, whether
one works with centralisers, multipliers, or multipliers at the $C_0^u(\G)$ level).

The second topic we look at is that of intrinsic groups, \cite{kn, cd}.
We can view the intrinsic group as an assignment, to each
quantum group, of a group.  Following \cite{kn} we can define the intrinsic
group using multipliers, and thus apply our previous study of morphisms to
multipliers.   We use this to show that the intrinsic group is a functor,
and again show that the different pictures of the intrinsic group, and the
different pictures of morphisms, interact very naturally.  Again, here it is
actually most natural to work with $C_0^u(\hh\G)$.  We give the new result that
the intrinsic group of $M(C_0^u(\hh\G))$, with the relative strict topology,
forms another natural equivalent definition of the intrinsic group of $\G$.

Let us say that $\G$ is ``classical'' if $C_0(\G) = C_0(G)$ for a locally
compact group $G$.  We show that the intrinsic group is the maximal
``classical subgroup'' of a
quantum group.  Closed subgroups of quantum groups were explored in \cite{dkss},
with two definitions offered.  We show that the intrinsic group is closed in
the stronger, ``Vaes closed'' sense.  We also recast this maximality property
as a universal property in the category theory sense, and show the the intrinsic
group functor is the adjoint of the inclusion functor from the category of
locally compact groups to the category of locally compact quantum groups.
A corollary of these results is that any classical, closed quantum subgroup is
Vaes closed, a result not available from \cite{dkss}.

Briefly, the organisation of the paper is as follows.  We firstly introduce
locally compact quantum groups, and then summarise the notion of a
``morphism''.  Compared to \cite{mrw} we work with ``left'' multiplicative
unitaries, and also ``reverse the arrows'', to better generalise from classical
group homomorphisms.  As such, we take a little time to give a detailed
summary, in the hope that this will be a useful reference.
In Section~\ref{sec:mults} we introduce centralisers and multipliers, and show
how to work with $C_0^u(\G)$ instead of $L^\infty(\G)$.  We then apply this
theory to show how multipliers and morphisms interact, finishing with some
comments on operator space structures and making links with the results of
\cite{jnr}.  In Section~\ref{sec:ig} we introduce the intrinsic group, again
lift the theory to $C_0^u(\G)$ (or actually the multiplier algebra, with
the strict topology).  This allows us to streamline some of the proofs from
\cite{kn}.  Then in Section~\ref{sec:igf} we show that the intrinsic group is
a functor, show the stated universal property, and then look at closed quantum
subgroups.

\subsection{Acknowledgements}

We would like to thank Matthias Neufang and Sutanu Roy for showing an
interest in this work at an early stage, and for asking about adjoint
functors.  We thank Stefaan Vaes for bringing the appendix of \cite{BV} to
our attention.  We thank the referee for useful comments.

\section{Locally compact quantum groups and their morphisms}

Let us introduce some notation.  We write $\otimes$ for the minimal tensor
product of $C^*$-algebras or the tensor product of Hilbert spaces, and write
$\vnten$ for the normal tensor product of
von Neumann algebras.  For a $C^*$-algebra $A$, we write $M(A)$ for the
multiplier algebra.  We shall use the theory of non-degenerate
$*$-homomorphisms, and their strict extensions, see \cite{lan} or
the appendix of \cite{mnw} for example.  We shall write $\mc B(H)$ for the
bounded linear operators on a Hilbert space $H$, and write $\mc B_0(H)$ for
the compact operators.  We use the standard theory of slice maps, and the
``leg numbering notation''.  We shall use $\sigma$ to denote the tensor
swap map, acting on Hilbert spaces or algebras.

We shall use basic, standard results from the theory of Operator Spaces,
specifically the notion of a Completely Bounded map, see \cite{ER}.

For the theory of locally compact quantum groups, we use \cite{kv,kvvn}.
For more gentle introductions, see \cite{kbook, vaesthesis}, and compare also
\cite{mnw}.  We follow the now
standard notation, and for a locally compact quantum group $\G$ write
$C_0(\G)$ for the $C^*$-algebra representing $\G$, and $L^\infty(\G)$ for the
von Neumann algebra, with $L^1(\G)$ its predual.  Writing $\Delta$ for the
coproduct, we can either view $\Delta$ as a non-degenerate $*$-homomorphism
$C_0(\G) \rightarrow M(C_0(\G)\otimes C_0(\G))$, or as a unital injective
normal $*$-homomorphism $L^\infty(\G) \rightarrow L^\infty(\G) \vnten
L^\infty(\G)$, which thus induces an algebra structure on $L^1(\G)$.
Similarly, $\Delta$ induces an algebra product on $C_0(\G)^*$.  We shall
denote these products by $\star$.

Locally compact quantum groups carry, by definition, invariant weights.
Use the left invariant weight to form the Hilbert space $L^2(\G)$.  Then
there is the fundamental multiplicative unitary $W$ on $L^2(\G)\otimes L^2(\G)$,
which implements the coproduct as
$\Delta(x) = W^*(1\otimes x)W$ for $x\in L^\infty(\G)$.  Then $C_0(\G)$
is weak$^*$-dense in $L^\infty(\G)$ and the inclusion $C_0(\G)\rightarrow
L^\infty(\G)$ repsects all of the associated maps.  One can start at either
the $C^*$-algebra level, \cite{kv}, or at the von Neumann algebra level,
\cite{kvvn}.  The set $\{ (\omega\otimes\id)(W) : \omega\in\mc B(L^2(G))_* \}
\subseteq \mc B(L^2(\G))$
is an algebra, and its closure a $C^*$-algebra, $C_0(\hh\G)$.  We can introduce
a coproduct on $C_0(\hh\G)$ by $\hh\Delta(x) = \hh W^*(1\otimes x)\hh W$,
where $\hh W = \sigma(W)^*$.  Then $C_0(\hh\G)$ can be given invariant weights
so as to become a locally compact quantum group, the weak$^*$-closure is
$L^\infty(\hh\G)$ which is the von Neumann algebraic version of $C_0(\hh\G)$,
and so we obtain $\hh\G$.  Then $W\in M(C_0(\G)\otimes C_0(\hh\G))
\subseteq L^\infty(\G)\vnten L^\infty(\hh\G)$.

The algebras $C_0(\G)$ and $L^\infty(\G)$ admit further maps, such as the
antipode, the unitary antipode, and the scaling group, but these will not play
a prominent role in this paper.

For a classical group $G$, we can form $C_0(G)$ and $C^*_r(G)$.  Alternatively,
we might form the universal group $C^*$-algebra $C^*(G)$.  There is a quantum
group analogue of this, \cite{kusuni}.  We shall write $C_0^u(\G)$ for the
``universal version'' of $C_0(\G)$; it is the enveloping $C^*$-algebra of a
certain $*$-subalgebra of $L^1(\hh\G)$.  We can lift $\Delta$, the antipode,
and so forth, to $C_0^u(\G)$.  We shall continue to write $\Delta$ for the
coproduct on $C_0^u(\G)$.  There is a quotient $*$-homomorphism
$\pi: C_0^u(\G) \rightarrow C_0(\G)$, the ``reducing morphism'', which intertwines
all of the associated maps.  Finally, we can lift $W$ to various universal
versions.  Here we depart from the notation of \cite{kusuni} and write
$\WW\in M(C_0^u(\G)\otimes C_0^u(\hh\G))$ for what is denoted by $\mc U$
in \cite{kusuni}, and then let $\Ww = (\id\otimes\hh\pi)(\WW)$ and
$\wW = (\pi\otimes\id)(\WW)$, which are denoted by $\mc V, \hh{\mc V}$,
respectively, in \cite{kusuni}.  An important result for us is that
\[ (\id\otimes\pi)\Delta(a) = \Ww^*(1\otimes\pi(a))\Ww
\qquad (a\in C_0^u(\G)), \]
see \cite[Proposition~6.2]{kusuni}.

As indicated here, when dealing with objects associated to $\hh\G$, we shall
adorn the object with a hat.  When dealing with more than one quantum group,
we shall adorn the objects with $\G$ or $\H$, and so forth, as appropriate.
For example, $W_\H$ is the fundamental unitary associated with $\H$, and
we have that $\hh W_\G = W_{\hh\G}$.  Note that in the notation of
\cite{kusuni}, we have that $\hh{\mc V}_\G = \sigma(\mc V_{\hh\G})^*$, which
explains our use of different notation.

Finally, we remark that a very similar theory can be built from just working
with special (``manageable'' or ``modular'') multiplicative unitaries, see
\cite{woro, sw}.  All of the results of this paper, excepting
Section~\ref{sec:closed}, hold in this more general setting.

\subsection{Morphisms of quantum groups}

A morphism $\G\rightarrow\H$ can be described in a number of equivalent
ways, \cite{mrw, ng}.  We shall work with ``left'' bicharacters, so as to more
closely match the conventions of Kustermans and Vaes.  We shall also
``reverse the arrows'', as compared to \cite{mrw}, so that a group homomorphism
$G\rightarrow H$ will give rise a morphism $\G\rightarrow\H$ if $C_0(\G) = 
C_0(G)$ and $ C_0(\H) = C_0(H)$ and not the other way around.  As such, we
shall give more detail in this summary than strictly necessary, in the hope it
will be a useful reference.

A morphism $\G\rightarrow\H$ can be equivalently described as:
\begin{enumerate}
\item A non-degenerate $*$-homomorphism $\phi:C_0^u(\H)\rightarrow M(C_0^u(\G))$
which intertwines the coproducts.  That is, a \emph{Hopf $*$-homomorphism}
between universal quantum groups.
\item As a \emph{bicharacter}, which is a unitary $U\in M(C_0(\G)
\otimes C_0(\hh\H))$ which satisfies $(\Delta_{\G}\otimes\id)(U)
= U_{13}U_{23}$ and $(\id\otimes\Delta_{\hh\H})(U) = U_{13}U_{12}$.
\item As a special type of coaction, termed a \emph{left quantum group
homomorphism}, which is a non-degenerate $*$-homomorphism $\beta:
C_0(\H) \rightarrow M(C_0(\G)\otimes C_0(\H))$ which is a coaction of $\G$,
namely $(\Delta_\G\otimes\id)\beta = (\id\otimes\beta)\beta$, and which
also satisfies $(\id\otimes\Delta_\H)\beta = (\beta\otimes\id)\Delta_\H$.
\end{enumerate}
Alternatively, we can express the second two conditions at the von Neumann
algebra level:
\begin{enumerate}
\item $U\in L^\infty(\G)\vnten L^\infty(\hh\H)$ unitary with the same conditions;
\item $\beta:L^\infty(\H) \rightarrow L^\infty(\G)\vnten L^\infty(\H)$
a normal unital $*$-homomorphism, with the same conditions.
\end{enumerate}
This can be shown by adapting the proofs of \cite{mrw}; compare also the
different proofs in \cite[Section~12]{kusuni}.  These different notions are
linked by the following properties:
\begin{enumerate}
\item Given $\phi$, we have that $U = (\pi_\G\phi\otimes\id)(\Ww_\H)$
and $\beta$ is the unique $*$-homomorphism which satisfies that
$\beta\pi_\H = (\pi_\G\phi\otimes\pi_\H)\Delta^u_\H$.
\item Given $U$, we define $\beta$ by
$\beta(x) = U^*(1\otimes x)U$ for $x\in C_0(\H)$ or $L^\infty(\H)$.
We can always ``lift'' $U$ to a bicharacter in $M(C_0^u(\G)\otimes C_0^u(\hh\H))$
and then the results of \cite[Section~6]{kusuni} readily give a unique
$\phi$ with $U = (\pi_\G\phi\otimes\id)(\Ww_\H)$.
\item Given $\beta$, there is a unique unitary $U$ with
$(\beta\otimes\id)(\ww_\H)\ww_{\H,23}^* = U_{13}$, and $U$ is a bicharacter.
From $U$ we obtain $\phi$, but there appears to be no simple, direct way to
move from $\beta$ to $\phi$.
\end{enumerate}

We remark that we also have the notion of a \emph{right quantum group
homomorphism}, but one can easily move between the left and right cases by
using the unitary antipodes of $\G$ and $\H$.  Furthermore, \cite{mrw} also
explores a fourth equivalent notion, namely that of a certain class of functors
between the categories of $C^*$-algebras with $C_0(\G)$-coactions and
$C_0(\H)$-coactions.

The identity morphism $\G\rightarrow\G$ is associated to the
bicharacter $W_\G$ and the quantum group homomorphism $\Delta_\G$.
We also remark that (left) quantum group homomorphisms $\beta$ are also
injective, and ``continuous'', meaning that the linear span of
$\beta(C_0(\H))(C_0(\G)\otimes 1)$ is contained in, and dense in,
$C_0(\G)\otimes C_0(\H)$.

Given a morphism $f:\G\rightarrow\H$, write $\phi_f, U_f, \beta_f$
for the associated objects above.  Given $f:\G\rightarrow\H$ and
$g:\H\rightarrow\K$, the composition $gf:\G\rightarrow\K$ is associated to:
\begin{itemize}
\item $\phi:C_0^u(\K)\rightarrow M(C_0^u(\G))$ given by $\phi = \phi_f\circ
\phi_g$ (the usual composition of non-degenerate $*$-homomorphisms).
\item $U\in M(C_0(\G)\otimes C_0(\hh\K))$ which is the unique unitary
satisfying $U_{g,23} U_{f,12} = U_{f,12} U_{13} U_{g,23}$ in
$\mc B(L^2(\G)\otimes L^2(\H)\otimes L^2(\K))$. 
\item $\beta:C_0(\K)\rightarrow M(C_0(\G)\otimes C_0(\K))$ which is the
unique non-degenerate $*$-homomorphism satisfying that
$(\beta_f\otimes\id)\beta_g = (\id\otimes\beta_g)\beta$.
\end{itemize}
We remark that it is possible to prove the existence of $U$ just using
$U_f,U_g$, and similarly construct $\beta$ just using $\beta_f,\beta_g$,
without having to pass through the various equivalences.

Finally, we mention duality.  For any morphism $\G\rightarrow\H$, there is
a ``dual morphism'' $\H\rightarrow\G$.  Writing $\phi,U$ for
$\G\rightarrow\H$ and $\hh\phi, \hh U$ for $\H\rightarrow\G$, we have
that:
\begin{itemize}
\item $\hh U = \sigma(U^*)$ where again $\sigma$ is the swap map;
\item $\phi$ and $\hh\phi$ are linked by the relation that
$(\phi\otimes\id)(\WW_\H) = (\id\otimes\hh\phi)(\WW_\G)$.
\end{itemize}
We remark that there appears to be no direct way to express the duality relation
at the level of coactions.  This is perhaps not surprising, as duality is
governed by the fundamental unitary, which is reflected in the bicharacter
picture, and of course used directly to relate $\phi$ and $\hh\phi$.

\section{Multipliers and morphisms}\label{sec:mults}

Completely bounded multipliers of locally compact quantum groups have been
studied in detail in \cite{jnr, hnr} and by the author in \cite{daws1, daws2}.
To be careful, we shall follow the notation of \cite{jnr}, but the
$C^*$-algebraic approach of \cite{daws2} will pay off here.

Fix a locally compact quantum group $\G$.  A \emph{left centraliser} of
$L^1(\G)$ is a right module map, that is, $L_*:L^1(\G)\rightarrow L^1(\G)$
with $L_*(\omega_1\star\omega_2) = L_*(\omega_1)\star\omega_2$.
Let $C^l_{cb}(L^1(\G))$ be the space of all completely bounded left
centralisers.  For $L_* \in C^l_{cb}(L^1(\G))$ let $L = (L_*)^* \in
\mc{CB}(L^\infty(\G))$ be the adjoint; that $L_*$ is a left centraliser is
then equivalent to $L$ being \emph{left covariant}, $(L\otimes\id)\Delta
= \Delta L$.

Recall the \emph{left regular representation} $\lambda:L^1(\G)\rightarrow
C_0(\hh\G)$, an injective algebra homomorphism defined by $\lambda(\omega)
= (\omega\otimes\id)(W)$.  We say that
$x\in L^\infty(\hh\G)$ is a \emph{left completely bounded multiplier} if
$x\lambda(\omega) \in \lambda(L^1(\G))$ for each $\omega\in L^1(\G)$, and the
resulting map $L^1(\G)\rightarrow L^1(\G)$ is completely bounded.  Denote this
by $x\in M_{cb}^l(L^1(\G))$.  By definition, $M_{cb}^l(L^1(\G)) \subseteq
C^l_{cb}(L^1(\G))$, and a major result of \cite{jnr}, see Corollary~4.4 of that
paper, is that every $L_*\in C^l_{cb}(L^1(\G))$ arises in this way.  We finessed 
this result in \cite[Theorem~4.2]{daws2}, showing that actually $M_{cb}^l(L^1(\G))
\subseteq M(C_0(\G))$; compare the self-contained approach of
\cite[Proposition~3.1]{daws1}.
We shall say that $x\in M_{cb}^l(L^1(\G))$ is ``associated to'' $L_*\in
C_{cb}^l(L^1(\G))$.  Notice that clearly $M_{cb}^l(L^1(\G))$ and $C_{cb}^l(L^1(\G))$
are isomorphic \emph{algebras} under this identification.

In this paper we shall work with left multipliers, but it is easy to see that
analogous results hold for right multipliers, either compare \cite{jnr,daws2},
or just work with the \emph{opposite quantum group}, to use the terminology
of \cite[Section~4]{kvvn}.
  We remark that double multipliers seem somewhat less well
understood, compare \cite[Section~7]{daws2}.

\subsection{Moving to the universal level}

In this section, we prove analogous results about the interaction of
$C_{cb}^l(L^1(G))$ and $C_0^u(\hh\G)$, in place of $C_0(\hh\G)$.  This then
allows us to study how morphisms and multipliers interact, as morphisms are
most naturally studied at the level of Hopf $*$-homomorphisms between universal
$C^*$-algebraic quantum groups.  The techniques here are inspired by
\cite{daws2}, though here we shall not work with Hilbert $C^*$-modules (but
one could construct analogous proofs using this machinery).

Let $L_*\in C^l_{cb}(L^1(\G))$.  We wish to consider $(L\otimes\id)(\wW)$,
but we need to make sense of this.  Let $C_0^u(\hh\G)$ be faithfully and
non-degenerated represented on a Hilbert space $K$, so we may identify
$M(C_0^u(\hh\G))$ with those $x\in\mc B(K)$ which multiply $C_0^u(\hh\G)$ into
itself; then $x\in C_0^u(\hh\G)''$.  Similarly, $M(C_0(\G)\otimes C_0^u(\hh\G))
\subseteq L^\infty(\G) \vnten C_0^u(\hh\G)'' \subseteq \mc B(L^2(\G)\otimes K)$.
Then $(L\otimes\id)(\wW) \in L^\infty(\G) \vnten C_0^u(\hh\G)''$ is well-defined.

\begin{lemma}\label{lem:in_mult_alg}
With the notation above, we have that $(L\otimes\id)(\wW) \in M(\mc B_0(L^2(\G))
\otimes C_0^u(\hh\G))$.
\end{lemma}
\begin{proof}
By the structure of \emph{normal} completely bounded maps and the structure
theory of normal $*$-homomorphisms, we can find a Hilbert space $H$ and
$S,T\in\mc B(L^2(\G),L^2(\G)\otimes H)$ with $\|S\| \|T\| = \|L\|_{cb}$ and
with $L(x) = S^*(x\otimes 1)T$ for all $x\in L^\infty(\G)$.  See, for example,
the discussion in \cite[Section~3]{daws1} or the proof of
\cite[Theorem~4.2]{daws2}.

Let $\theta\in\mc B_0(L^2(\G))$ and $a\in C_0^u(\hh\G)$.  Let $U:L^2(\G)
\rightarrow L^2(\G)\otimes H$ be some isometry, and set $R = T\theta U^*$ so
that $RU=T\theta$.  Then, working in
$L^\infty(\G) \vnten C_0^u(\hh\G)'' \subseteq \mc B(L^2(\G)\otimes K)$,
\begin{align*}
(L\otimes\id)(\wW) (\theta\otimes a)
&= (S\otimes 1)^* \wW_{13} (T\theta\otimes a)
= (S\otimes 1)^* \wW_{13} (R\otimes a)(U\otimes 1).
\end{align*}
Now, $R \in \mc B_0(L^2(\G)\otimes H)$ because $\theta$ is compact.
As $\wW\in M(\mc B_0(L^2(\G)) \otimes C_0^u(\hh\G))$, and using that
$\mc B_0(L^2(\G)\otimes H) = \mc B_0(L^2(\G)) \otimes \mc B_0(H)$,
it follows that
\[ \wW_{13} (R\otimes a) \in \mc B_0(L^2(\G)\otimes H) \otimes C_0^u(\hh\G). \]
Hence $(L\otimes\id)(\wW) (\theta\otimes a) \in
\mc B_0(L^2(\G)) \otimes C_0^u(\hh\G)$.  Analogously we can show this with
the order of products reversed, and so $(L\otimes\id)(\wW) \in M(\mc B_0(L^2(\G))
\otimes C_0^u(\hh\G))$ as claimed.
\end{proof}

We shall use the ``invariants are constant'' idea from \cite{mrw}, compare
\cite[Theorem~2.1]{daws1} and \cite[Lemma~4.6]{aristov}: if $x,y \in
L^\infty(\G)$ with $\Delta(x) = y\otimes 1$ then $x=y\in\mathbb C1$.
In the following, $\lambda_u:L^1(\G)\rightarrow C_0^u(\hh\G)$ is the
``universal left-regular representation'',
$\omega\mapsto (\omega\otimes\id)(\wW)$.

\begin{theorem}\label{thm:cent_is_mult}
Let $L_* \in C_{cb}^l(L^1(\G))$.  There exists $x\in M(C_0^u(\hh\G))$
with $x \lambda_u(\omega) = \lambda_u(L_*(\omega))$ for $\omega\in L^1(\G)$,
or equivalently, $(L\otimes\id)(\wW) = (1\otimes x)\wW$.

Furthermore, if $L_*\in\mc{CB}(L^1(\G))$ is any completely bounded map
and $x\in M(C_0^u(\hh\G))$ any element, such
that $(L\otimes\id)(\wW) = (1\otimes x)\wW$, then $L_*$ is a centraliser,
associated with $\pi_{\hh\G}(x) \in M_{cb}^l(L^1(\G))$.
\end{theorem}
\begin{proof}
Again working in $L^\infty(\G) \vnten C_0^u(\hh\G)'' \subseteq
\mc B(L^2(G)\otimes K)$, we set $X = (L\otimes\id)(\wW) \wW^* \in
L^\infty(\G) \vnten C_0^u(\hh\G)''$.  By the previous lemma, we know that
$X\in M(\mc B_0(L^2(\G)) \otimes C_0^u(\hh\G))$.  We calculate that
\begin{align*} (\Delta\otimes\id)(X)
&= (\Delta L\otimes\id)(\wW) (\wW_{13} \wW_{23})^*
= ((L\otimes\id)\Delta\otimes\id)(\wW) \wW_{23}^* \wW_{13}^* \\
&= (L\otimes\id\otimes\id)(\wW_{13} \wW_{23}) \wW_{23}^* \wW_{13}^*
= ( (L\otimes\id)(\wW) \wW^*)_{13} = X_{13}. \end{align*}
By slicing on the right, and using that invariants are constant, we see that
$X \in \mathbb C 1 \vnten M(C_0^u(\hh\G))$.  Thus, there is $x\in M(C_0^u(\hh\G))$
with $X = 1\otimes x$, or equivalently,
\[ (L\otimes\id)(\wW) = (1\otimes x)\wW. \]

Now suppose that all we know about $L$ is that $(L\otimes\id)(\wW)
= (1\otimes x)\wW$.  Then
\[ \lambda_u(L_*(\omega)) = (\omega\otimes\id)((L\otimes\id)(\wW))
= (\omega\otimes\id)((1\otimes x)\wW)
= x \lambda_u(\omega), \]
and so, as $\lambda_u$ is injective, $L_*$ is a centraliser.
As $\pi_{\hh\G} \lambda_u = \lambda$, it follows
that $\pi_{\hh\G}(x) \in M_{cb}^l(L^1(\G))$ is the multiplier given by $L_*$.
\end{proof}

\begin{definition}
Let the collection of those $x\in M(C_0^u(\hh\G))$ associated to completely
bounded left centralisers be denoted by $M_{cb}^{l,u}(L^1(\G))$.
\end{definition}

We note that (the strict extension of) $\pi_{\hh\G}$ restricts to a bijection
$M_{cb}^{l,u}(L^1(\G)) \rightarrow M_{cb}^{l}(L^1(\G))$.

While discussing multiplier algebras, we make the following remark.
Let $L_*\in C^l_{cb}(L^1(\G))$.
As elements of the form $(\id\otimes\omega)(W)$ are norm dense
in $C_0(\G)$, and as $(L\otimes\id)(W) = (1\otimes x)W$, it follows
that $L$ restricts to a map $C_0(\G)\rightarrow C_0(\G)$.  We remark that we
don't know if $L$ gives a map $M(C_0(\G))\rightarrow M(C_0(\G))$.

We next adapt an idea from the proof of \cite[Theorem~4.2]{daws2}, which will
allow us to find a ``universal $C^*$-algebraic'' version of the representation
theorem of \cite{jnr}.  Firstly, we recall some notions from, for example,
\cite[Section~5.5]{kv}.  For a $C^*$-algebra $A$ and an index set $I$, let
$MC_I(A)$ be the families $(x_i)_{i\in I} \subseteq M(A)$ with
$\sum_i x_i^*x_i$ strictly converging in $M(A)$.  Similarly let $MR_I(A)$ be
those $(x_i)$ with $\sum_i x_ix_i^*$ strictly converging, so $(x_i)\in MC_I(A)$
if and only if $(x_i^*)\in MR_I(A)$.  For $x,y\in MC_I(A)$,
\cite[Lemma~5.28]{kv} shows that $\sum_i x_i^*y_i$ is strictly convergent and
that the partial sums form a bounded family.  Furthermore, these notions are
stable under applying non-degenerate $*$-homomorphisms.  These notions have
obvious links with the (extended) Haagerup tensor product, see for example
\cite[Proposition~1.15]{vvd}.

\begin{theorem}\label{thm:mci_form}
For $x\in M_{cb}^{l,u}(L^1(\G))$, there exists $(a_i),(b_i)\in MC_I(C_0^u(\hh\G))$
with $\sum_i (1\otimes b_i^*)\Delta(a_i) = x\otimes 1$.
\end{theorem}
\begin{proof}
Let $L_*\in C_{cb}^l(L^1(\G))$ and $x$ be linked as before.
Let $L(a) = S^*(a\otimes 1)T$ as in the proof of Lemma~\ref{lem:in_mult_alg},
and again suppose that $C_0^u(\hh\G)$ is represented on a Hilbert space $K$.
Fix a unit vector $\xi\in L^2(\G)$, let $(e_i)$ be an orthonormal basis of
$L^2(\G)\otimes H$, and for each $i$, set
\[ a_i = (\id\otimes\omega_{\xi,e_i})(\hh\Ww^*_{13}(1\otimes T)\hh\Ww). \]
This is a slight abuse of notation; by definition, what we mean is that
$a_i\in\mc B(K)$ is the operator
\[ (a_i(\alpha)|\beta) = \big( \hh\Ww^*_{13}(1\otimes T)\hh\Ww
(\alpha\otimes\xi) \big| \beta\otimes e_i \big)
\qquad (\alpha,\beta\in K), \]
which makes sense as $\Ww \in M(C_0^u(\hh\G)\otimes C_0(\G))
\subseteq \mc B(K\otimes L^2(\G))$.  Equivalently,
\[ \sum_i a_i(\alpha)\otimes e_i = \hh\Ww^*_{13}(1\otimes T)\hh\Ww
(\alpha\otimes\xi)  \qquad (\alpha\in K). \]
Thus, for $\alpha,\beta\in K$,
\begin{align*} \Big( \sum_i a_i^*a_i (\alpha) \Big| \beta \Big)
&= \big( \hh\Ww^*_{13}(1\otimes T)\hh\Ww (\alpha\otimes\xi) \big|
\hh\Ww^*_{13}(1\otimes T)\hh\Ww(\beta\otimes\xi) \big) \\
&= \big( \hh\Ww^*(1\otimes T^*T)\hh\Ww(\alpha\otimes\xi) \big|
   \beta\otimes\xi \big). \end{align*}
So $\sum_i a_i^*a_i = (\id\otimes\omega_{\xi,\xi})(\hh\Ww^*(1\otimes T^*T)\hh\Ww)
\in M(C_0^u(\hh\G))$, where the sum converges weakly in $\mc B(K)$.
We claim that each $a_i\in M(C_0^u(\hh\G))$, and that this sum converges strictly.
We shall show that for $a\in C_0^u(\hh\G)$, the sum
$\sum_i a_i^*a_ia$ is norm convergent; the proof for
$\sum_i aa_i^*a_i$ is similar.

Choose $\theta\in\mc B_0(L^2(\G))$ with $\theta(\xi)=\xi$.
As $\hh\Ww(a\otimes\theta)\in C_0^u(\hh\G)\otimes\mc B_0(L^2(\G))$,
and arguing similarly to the proof of Lemma~\ref{lem:in_mult_alg},
for $\epsilon>0$ we can find $(a_j)_{j=1}^n \subseteq C_0^u(\hh\G)$ and
$(\xi_j)_{j=1}^n \subseteq L^2(\G)\otimes H$ with
\[ \Big\| \hh\Ww^*_{13}(1\otimes T)\hh\Ww(a\alpha\otimes\xi) 
- \sum_j a_j(\alpha) \otimes \xi_j \Big\| \leq \epsilon\|\alpha\|
\qquad (\alpha\in K). \]
Here we used that $\hh\Ww(a\alpha\otimes\xi) =
\hh\Ww(a\otimes\theta)(\alpha\otimes\xi)$ for any $\alpha$.
For each $j$ let $\xi_j = \sum_i \xi_{j,i} e_i$, so equivalently we have that
\[ \Big\| \sum_i \Big( a_ia(\alpha)
- \sum_j \xi_{j,i}a_j(\alpha) \Big) \otimes e_i \Big\| \leq \epsilon\|\alpha\|
\qquad (\alpha\in K). \]
From this (and a similar argument with $a$ on the left) it follows that
indeed $a_i\in M(C_0^u(\hh\G))$ for each $i$.  Continuing,
let $\epsilon_0>0$ with $\epsilon_0 \sum_j \|a_j\| \leq \epsilon$, and
choose a finite subset
$I_0\subseteq I$ with, for each $j$,
\[ \sum_{i\not\in I_0} |\xi_{j,i}|^2 \leq \epsilon_0^2. \]
Then
\begin{align*}
\Big\| \sum_{i\not\in I_0} a_ia(\alpha) \otimes e_i \Big\|
&= \Big\| \sum_i a_ia(\alpha) \otimes e_i -
   \sum_{i\in I_0} a_ia(\alpha) \otimes e_i \Big\| \\
&\leq \epsilon\|\alpha\| + \Big\| \sum_i \sum_j \xi_{j,i}a_j(\alpha) \otimes e_i
   - \sum_{i\in I_0} a_ia(\alpha) \otimes e_i \Big\| \\
&\leq \epsilon\|\alpha\| + \Big\| \sum_{i\not\in I_0}
   \sum_j \xi_{j,i}a_j(\alpha) \otimes e_i \Big\| +
   \Big\| \sum_{i\in I_0} \Big( \sum_j \xi_{j,i}a_j(\alpha) - a_ia(\alpha)
   \Big) \otimes e_i \Big\| \\
&\leq \epsilon\|\alpha\| + \|\alpha\| \epsilon_0 \sum_j \|a_j\|
   + \epsilon \|\alpha\|
\leq 3 \epsilon\|\alpha\|.
\end{align*}
Finally, we then see that for $\alpha,\beta\in K$,
\begin{align*}
\Big| \Big(\sum_{i\in I_0} a_i^*a_i a(\alpha)\Big|\beta\Big)
&- \big( \hh\Ww^*(1\otimes T^*T)\hh\Ww(a(\alpha)\otimes\xi) \big|
   \beta\otimes\xi\big) \Big| \\
&= \Big| \Big( \sum_{i\in I_0} a_ia(\alpha)\otimes e_i \Big|
\sum_j a_j(\beta)\otimes e_j \Big)
- \Big( \sum_i a_ia(\alpha)\otimes e_i \Big|
\sum_j a_j(\beta)\otimes e_j \Big) \Big| \\
&\leq \|T\| \|\beta\| \Big\| \sum_{i\not\in I_0} a_ia(\alpha) \otimes e_i \Big\|.
\end{align*}
Putting these together, we see that
\[ \Big\| \sum_{i\in I_0} a_i^*a_i a -
(\id\otimes\omega_{\xi,\xi})(\hh\Ww^*(1\otimes T^*T)\hh\Ww)a \Big\|
\leq 3\epsilon, \]
as required.  Thus $(a_i) \in MC_I(C_0^u(\hh\G))$.

We similarly define $b_i = 
(\id\otimes\omega_{\xi,e_i})(\hh\Ww^*_{13}(1\otimes S)\hh\Ww)$.
Now, using a similar abuse of notation,
\[ \Delta(a_i) = (\id\otimes\id\otimes\omega_{\xi,e_i})\big(
\hh\Ww_{24}^*\hh\Ww_{14}^*(1\otimes 1\otimes T)\hh\Ww_{13}\hh\Ww_{23} \big). \]
Then
\begin{align*} \sum_i & (1\otimes b_i)^*\Delta(a_i)
= \sum_i (\id\otimes\id\otimes\omega_{e_i,\xi})\big(\hh\Ww^*_{23}
(1\otimes 1\otimes S^*) \hh\Ww_{24}\big) \\
&\qquad\qquad\qquad\qquad\qquad
(\id\otimes\id\otimes\omega_{\xi,e_i})\big(
\hh\Ww_{24}^*\hh\Ww_{14}^*(1\otimes 1\otimes T)\hh\Ww_{13}\hh\Ww_{23} \big) \\
&= (\id\otimes\id\otimes\omega_{\xi,\xi})\big(\hh\Ww^*_{23}
(1\otimes 1\otimes S^*) \hh\Ww_{24}
\hh\Ww_{24}^*\hh\Ww_{14}^*(1\otimes 1\otimes T)\hh\Ww_{13}\hh\Ww_{23} \big) \\
&= (\id\otimes\id\otimes\omega_{\xi,\xi})\big(\hh\Ww^*_{23}
(1\otimes 1\otimes S^*) \hh\Ww_{14}^*(1\otimes 1\otimes T)\hh\Ww_{13}
\hh\Ww_{23} \big) \\
&= (\id\otimes\id\otimes\omega_{\xi,\xi})\big(\hh\Ww^*_{23}
(\id\otimes L)(\hh\Ww^*)_{13}\hh\Ww_{13}
\hh\Ww_{23} \big).
\end{align*}
As $\hh\Ww^* = \sigma( \wW )$ and $(L\otimes\id)(\wW) = (1\otimes x)\wW$,
it follows that
\[ \sum_i (1\otimes b_i)^*\Delta(a_i)
= (\id\otimes\id\otimes\omega_{\xi,\xi})\big(\hh\Ww^*_{23}
(x\otimes 1)_{13}(\hh\Ww^*)_{13}\hh\Ww_{13}
\hh\Ww_{23} \big)
= x\otimes 1, \]
as claimed.
\end{proof}

We now state a converse to the previous result, and show how to recover the
original centraliser.

\begin{theorem}\label{thm:and_converse}
Let $(a_i),(b_i)\in MC_I(C_0^u(\hh\G))$ be such that
$\sum_i (1\otimes b_i^*)\Delta(a_i) = x\otimes 1$ for some $x\in M(C_0^u(\hh\G))$.
Then $x\in M_{cb}^{l,u}(L^1(\G))$ and the associated
$L_*\in C_{cb}^l(L^1(\G))$ is given by
$L(a) = \sum_i \pi_{\hh\G}(b_i)^*  a \pi_{\hh\G}(a_i)$ for $a\in L^\infty(\G)$,
with convergence weakly in $\mc B(L^2(\G))$.
\end{theorem}
\begin{proof}
Define $L:L^\infty(\G)\rightarrow\mc B(L^2(\G)); a \mapsto \sum_i
\pi_{\hh\G}(b_i)^*a\pi_{\hh\G}(a_i)$, so that $L$ is a normal
completely bounded map.  Consider then
\begin{align*} (L\otimes\id)(\wW)
&= \sigma \sum_i (1\otimes\pi_{\hh\G}(b_i)^*)\hh\Ww^*
(1\otimes \pi_{\hh\G}(a_i)) \\
&= \sigma \sum_i (1\otimes\pi_{\hh\G}(b_i)^*)\hh\Ww^*
(1\otimes \pi_{\hh\G}(a_i)) \hh\Ww \hh\Ww^* \\
&= \sigma \sum_i (1\otimes\pi_{\hh\G}(b_i)^*)
   (\id\otimes\pi_{\hh\G})\Delta(a_i) \hh\Ww^* \\
&= \sigma\Big( (\id\otimes\pi_{\hh\G})
   \Big(\sum_i (1\otimes b_i^*) \Delta(a_i) \Big) \hh\Ww^*\Big) \\
&= (1\otimes x) \wW.
\end{align*}
By applying $\id\otimes\pi_{\hh\G}$ we also see that $(L\otimes\id)(W)
= (1\otimes\pi_{\hh\G}(x))(W)$.
As slices $(\id\otimes\hh\omega)(W)$, with $\hh\omega \in L^1(\hh\G)$,
form a weak$^*$-dense subspace of
$L^\infty(\G)$, this calculation shows that $L$ does indeed map $L^\infty(\G)$
to $L^\infty(\G)$.  Furthermore, we have now verified the condition in
Theorem~\ref{thm:cent_is_mult} and so $L_*$ is a left completely bounded
centraliser, associated to the ``universal'' multiplier $x$, which completes
the proof.
\end{proof}

\subsection{Multipliers and morphisms}\label{sec:mult_mor}

Throughout this section, let $\G,\H$ be locally compact quantum groups,
and let $\G\rightarrow\H$
be a morphism, represented by $\phi:C_0^u(\H)\rightarrow M(C_0^u(\G))$,
$U\in M(C_0(\G)\otimes C_0(\hh\H))$ and $\beta:C_0(\H)\rightarrow M(C_0(\G)
\otimes C_0(\H))$.  Let $\hh\phi:C_0^u(\hh\G)\rightarrow M(C_0^u(\hh\H))$ be
the dual Hopf $*$-homomorphism.

The following is then the most natural way that morphisms and multipliers could
interact in.  We shall then go on to show how the other ``pictures'' also
interact in natural ways.

\begin{theorem}\label{thm:main_mor}
The map $\hh\phi$ restricts to a homomorphism $M_{cb}^{l,u}(L^1(\G))
\rightarrow M_{cb}^{l,u}(L^1(\H))$.
\end{theorem}
\begin{proof}
Let $L_*\in C_{cb}^l(L^1(\G))$ and let $(a_i),(b_i)\in MC_I(C_0^u(\hh\G))$
as in Theorem~\ref{thm:mci_form}, so $L_*$ is associated to
$x\in M_{cb}^{l,u}(L^1(\G))$ where
$\sum_i (1\otimes b_i^*)\Delta(a_i) = x\otimes 1$.

As $\hh\phi$ is a Hopf $*$-homomorphism,
\begin{align*}
\sum_i (1\otimes \hh\phi(b_i^*))\Delta(\hh\phi(a_i))
= (\hh\phi\otimes\hh\phi)\sum_i (1\otimes b_i^*)\Delta(a_i)
= \hh\phi(x)\otimes 1, \end{align*}
and so an application of Theorem~\ref{thm:and_converse} shows that there
is $L' \in C_{cb}^l(L^1(\H))$ associated with $\hh\phi(x)\in
M_{cb}^{l,u}(L^1(\H))$ as required.
\end{proof}

\begin{remark}
The ``classical'' situation here is detailed in \cite[Section~6.1]{spronk},
where it is shown that a group homomorphism $G\rightarrow H$, which induces
a Hopf $*$-homomorphism $C_0(H) \rightarrow C_b(G)$, restricts to a map
$M_{cb}A(H) \rightarrow M_{cb}A(G)$.  In our language, we would start with
a morphism $\hh H \rightarrow \hh G$, say given by $\phi:C^*(G) \rightarrow
M(C^*(H))$, and then consider the dual $\hh\phi:C_0(H) \rightarrow C_b(G)$.
Hence we exactly recover the classical result, once we have the ``duality
convention'' correct.  
\end{remark}

One way to find centralisers of $L^1(\G)$ is to embed $L^1(\G)$ into
$C_0^u(\G)^*$, where it becomes a closed two-sided ideal, and so (left)
multiplication by elements of $C_0^u(\G)^*$ define members of
$C_{cb}^l(L^1(\G))$ (and all \emph{completely positive} centralisers arise in
this way, \cite{daws1}).  The following shows that morphisms, from the Hopf
$*$-homomorphism perspective, behave as expected.

\begin{proposition}\label{prop:comm_diag}
We have the commutative diagram
\[ \xymatrix{ C_0^u(\G)^* \ar[r]^-{\phi^*} \ar[d] & C_0^u(\H)^* \ar[d] \\
M_{cb}^{l,u}(L^1(\G)) \ar[r]^-{\hh\phi} & M_{cb}^{l,u}(L^1(\H)) } \]
where the bottom arrow is given by the previous theorem.
\end{proposition}
\begin{proof}
Let $\mu\in C_0^u(\G)^*$ and write $\pi_\G^*: L^1(\G) \rightarrow C_0^u(\G)^*$
for the embedding, which is a completely isometric algebra homomorphism.
Let $\mu$ induce $L_*\in C_{cb}^l(L^1(\G))$, which means that
$\pi_\G^*(L_*(\omega))
= \mu \pi_\G^*(\omega) = (\mu\otimes\omega)(\id\otimes\pi_\G)\Delta$
for each $\omega\in L^1(\G)$.  Then
\begin{align*} (L\otimes\id)(\wW) &= (L\pi_\G\otimes\id)(\WW)
= ((\mu\otimes\pi_\G)\Delta\otimes\id)(\WW)
= (\mu\otimes\pi_\G\otimes\id)(\WW_{13} \WW_{23}) \\
&= ((\mu\otimes\id)(\WW)\otimes 1) \wW. \end{align*}
By Theorem~\ref{thm:cent_is_mult}, we see that $x = (\mu\otimes\id)(\WW)
\in M(C_0^u(\hh\G)) \in M_{cb}^{l,u}(L^1(\G))$ is associated to $L$ and
hence to $\mu$.

Similarly let $\phi^*(\mu) \in C_0^u(\H)^*$ induce $L'_*\in C_{cb}^l(L^1(\H))$
which is thus associated to $x'\in M_{cb}^{l,u}(L^1(\H))$ where
\[ x' = (\phi^*(\mu)\otimes\id)(\WW_\H)
= (\mu\otimes\id) \big( (\phi\otimes\id)(\WW_\H) \big)
= (\mu\otimes\id) \big( (\id\otimes\hh\phi)(\WW_\G) \big)
= \hh\phi(x), \]
as required to show that the diagram commutes.
\end{proof}

We now demonstrate a similar link at the level of centralisers, and not
multipliers, using bicharacters and quantum group homomorphisms (a picture
not really explored in \cite{spronk}, for example, but see
Remark~\ref{rem:bichar_only} below for links with \cite{kal}).

\begin{lemma}\label{lem:cent_bichar}
Let $L_*\in C_{cb}^l(L^1(\G))$ and $x\in M_{cb}^{l,u}(L^1(\G))$ be linked,
and let $U$ be the bicharacter representing the morphism $\G\rightarrow\H$.
Then we have that $(L\otimes\id)(U) = (1\otimes \pi_{\hh\H}\hh\phi(x))U$.
\end{lemma}
\begin{proof}
As $U = (\pi_\G\phi\otimes\id)(\Ww_\H)$ we see that
\begin{align*} (L\otimes\id)(U) &= (L\pi_\G\phi\otimes\pi_{\hh\H})(\WW_\H)
= (L\pi_\G\otimes\pi_{\hh\H}\hh\phi)(\WW_\G)
= (L\otimes\pi_{\hh\H}\hh\phi)(\wW_\G) \\
&= (\id\otimes\pi_{\hh\H}\hh\phi)\big((1\otimes x)\wW_\G\big)
= (1\otimes \pi_{\hh\H}\hh\phi(x)) (\id\otimes\pi_{\hh\H}\hh\phi)(\wW_\G) \\
&= (1\otimes \pi_{\hh\H}\hh\phi(x)) U,
\end{align*}
as claimed.
\end{proof}

\begin{proposition}\label{prop:for_coactions}
Let $L_*\in C_{cb}^l(L^1(\G))$ be mapped to $L'_*\in C_{cb}^l(L^1(\H))$ by
the morphism $\G\rightarrow\H$, and let $U$ be the bicharacter representing
this morphism.  Then
\[ 1\otimes L'(a) = U (L\otimes\id)(U^*(1\otimes a)U) U^*
\qquad (a\in L^\infty(\H). \]
\end{proposition}
\begin{proof}
By weak$^*$-continuity, it is enough to show this for
$a=(\id\otimes\omega)(W_\H)$.  With this in mind, the claim is equivalent to
\[ 1\otimes (L'\otimes\id)(W_\H) = U_{12}(L\otimes\id\otimes\id)
( U^*_{12} W_{\H,23} U_{12} ) U^*_{12}. \]
Now, we have that $W_{\H,23} U_{12} W_{\H,23}^* = U_{12} U_{13}$ as $U$ is
a bicharacter, and so the right hand side equals, using the previous lemma,
\begin{align*}
U_{12}(L\otimes\id\otimes\id)( U_{13} W_{\H,23} ) U^*_{12}
&= U_{12} \big( (1\otimes \pi_{\hh\H}\hh\phi(x)) U \big)_{13}
   W_{\H,23} U^*_{12} \\
&= U_{12} (1\otimes 1\otimes \pi_{\hh\H}\hh\phi(x)) U_{13} W_{\H,23} U^*_{12} \\
&= U_{12} (1\otimes 1\otimes \pi_{\hh\H}\hh\phi(x)) U^*_{12} W_{\H,23}  \\
&= 1 \otimes (1\otimes \pi_{\hh\H}\hh\phi(x))(W_\H)
= 1\otimes (L'\otimes\id)(W_\H)
\end{align*}
as claimed.
\end{proof}

Recalling that the quantum group homomorphism $\beta$ satisfies that
$\beta(a) = U^*(1\otimes a)U$ for $a\in L^\infty(\H)$, the following is
immediate.  Notice that we can think of this as being a generalisation of
the covariance condition which defines what it means for $L$ to be (the
adjoint of) a centraliser.

\begin{corollary}\label{corr:for_coactions}
Let $L_*\in C_{cb}^l(L^1(\G))$ be mapped to $L'_*\in C_{cb}^l(L^1(\H))$ by
our morphism, which is represented by the quantum group homomorphism $\beta$.
Then $\beta L' = (L\otimes\id) \beta$.
\end{corollary}

This makes immediate sense if we work at the von Neumann algebra level, and
regard $\beta$ as a map $L^\infty(\G)\rightarrow L^\infty(\G)\vnten L^\infty(\H)$.  It is not clear how to, \emph{a priori}, give a purely
$C^*$-algebraic interpretation of this.

\begin{remark}\label{rem:bichar_only}
It is possible to work purely at the level of bicharacters and centralisers, 
without passing to multipliers and Hopf $*$-homomorphisms.  Indeed, let $U
\in L^\infty(\G) \vnten L^\infty(\hh\H)$ represent $\G\rightarrow\H$, and let
$L_*\in C_{cb}^l(\G)$.  Then consider $(L\otimes\id)(U)U^*$.  By applying
$(\Delta\otimes\id)$, and arguing as in the proof of
Theorem~\ref{thm:cent_is_mult}, we find $x\in L^\infty(\hh\H)$ with
$(L\otimes\id)(U) = (1\otimes x)U$.
Of course, $x$ will turn out to be the multiplier associated to $L'$.

For $a\in L^\infty(\H)$, we can now consider
$U(L\otimes\id)(U^*(1\otimes a)U)U^*$.
The proof of Proposition~\ref{prop:for_coactions} still works, and we find
that if $a=(\id\otimes\omega)(W_\H)$ then there is $L'(a)\in L^\infty(\H)$
with $U^*(1\otimes L'(a))U = (L\otimes\id)(U^*(1\otimes a)U)$.  Indeed,
$L'(a) = (\id\otimes\omega x)(W_\H)$.  By normality, it follows easily that
$L'$ extends to a completely bounded normal map $L^\infty(\H)\rightarrow
L^\infty(\H)$, and also $(L'\otimes\id)(W_\H) = (1\otimes x)W_\H$ and so
$L'$ is the adjoint of a centraliser.
\end{remark}

This argument, and Proposition~\ref{prop:for_coactions}, should also 
be compared with \cite[Theorem~2.1]{kal}, where the relation between
centralisers and actions of quantum groups (at the von Neumann algebra level)
is explored: we can apply this to $\beta$, as $\beta$ is a (special sort of)
coaction.  Notice that our use of the ``invariants are constant'' approach
allows us to avoid weight theory, and the use of crossed product theory.

Let us finally make some remarks about operator space structures.
The space $C_{cb}^l(L^1(\G))$ inherits a natural operator space structure as
a subspace of $\mc{CB}(L^1(\G)) \subseteq \mc{CB}(L^\infty(\G))$, and using
this, we induce an operator space structure on $M_{cb}^l(L^1(\G))$ and
$M_{cb}^{l,u}(L^1(\G))$.  As $\beta$ is a complete isometry, it follows more
or less immediately from Corollary~\ref{corr:for_coactions} that the map
$L\mapsto L'$ is a complete contraction.  Let us formally state this.

\begin{theorem}
A morphism $\G\rightarrow\H$ induces a complete contraction
$C_{cb}^l(L^1(\G)) \rightarrow C_{cb}^l(L^1(\H))$ and thus a
complete contraction $M_{cb}^{l,u}(L^1(\G)) \rightarrow M_{cb}^{l,u}(L^1(\H))$.
\end{theorem}

\subsection{The representation theorem}

Let $\mc{CB}^{\sigma,L^\infty(\G)}_{L^\infty(\hh\G)'}
(\mc B(L^2(\G)))$ be the space of normal completely bounded maps
$\Phi:\mc B(L^2(\G)) \rightarrow \mc B(L^2(\G))$ which restrict to maps
$L^\infty(\G)\rightarrow L^\infty(\G)$, and which are
$L^\infty(\hh\G)'$-bimodule maps.  The paper \cite{jnr} shows that
$M_{cb}^l(L^1(\G))$ is (completely isometrically) isomorphic to this space.

Given $L_*\in C_{cb}^l(L^\infty(\G))$ we can extend $L$ to all of
$\mc B(L^2(\G))$ in such a way that $L$ becomes a $L^\infty(\hh\G)'$-bimodule
map.  Indeed, we claim that for each $x\in\mc B(L^2(\G))$ there is $\Phi(x)
\in \mc B(L^2(\G))$ with
\[ 1\otimes\Phi(x) = W\big( (L\otimes\id)(W^*(1\otimes x)W) \big) W^*. \]
Then $\Phi$ is easily seen to be completely bounded, normal, to extend $L$,
and to be a $L^\infty(\hh\G)'$-bimodule map, as $W\in L^\infty(\G)\vnten
L^\infty(\hh\G)$.  That $\Phi$ exists can shown using the ``invariants are
constant'' technique, see \cite[Proposition~3.2]{daws1}.  Here we shall follow
the original approach of \cite{jnr}, and use that the linear span of
$\{ ab : a\in L^\infty(\G), b\in L^\infty(\hh\G)' \}$ is weak$^*$-dense in
$\mc B(L^2(\G))$, see for example \cite[Theorem~2.2]{daws1}.  For such $ab$
we see that
\[ W\big( (L\otimes\id)(W^*(1\otimes ab)W) \big) W^*
= W\big( (L\otimes\id)\Delta(a) \big) W^* (1\otimes b)
= 1\otimes L(a)b. \]
Thus $\Phi$ exists, $\Phi(ab) = L(a)b$ and similarly $\Phi(ba) = bL(a)$,
so establishing all the needed properties.  That any $\Phi\in
\mc{CB}^{\sigma,L^\infty(\G)}_{L^\infty(\hh\G)'}(\mc B(L^2(\G)))$ arises in
exactly this way is more intricate, see \cite[Theorem~4.10]{jnr}.

\begin{proposition}
Continuing with this notation, let $L_*\in C_{cb}^l(L^1(\G))$ be extended to
$\Phi$, and using a morphism $\G \rightarrow \H$, let $L_*$ be mapped to
$L'_*\in C_{cb}^l(L^1(\H))$, which is extended to $\Phi'$.  Then
$U^*(1\otimes\Phi'(x))U = (\Phi\otimes\id)(U^*(1\otimes x)U)$ for all
$x\in\mc B(L^2(\H))$ where again $U$ is the bicharacter associated to our
morphism.
\end{proposition}
\begin{proof}
By weak$^*$-continuity, it suffices to verify this for $x=ab$ with
$a\in L^\infty(\H)$ and $b\in L^\infty(\hh\H)'$.  However, as $U\in L^\infty(\G)
\vnten L^\infty(\hh\H)$, we see that $U^*(1\otimes x)U \in L^\infty(\G) \vnten
\mc B(L^2(\H))$ and so
\begin{align*}
(\Phi\otimes\id)(U^*(1\otimes x)U)
&= (L\otimes\id)(U^*(1\otimes a)U)(1\otimes b)
= U^*(1\otimes L'(a))U (1\otimes b) \\
&= U^*(1\otimes L'(a)b)U
= U^*(1\otimes \Phi'(x))U, \end{align*}
using Proposition~\ref{prop:for_coactions} and the discussion above.
\end{proof}

\section{Intrinsic groups}\label{sec:ig}

The \emph{intrinsic group} of a Kac algebra was stuided by De Canni{\'e}re in
\cite{cd} (for example), and for locally compact quantum groups by Kalantar and
Neufang in \cite{kn}.  In this section, we will show that the assignment of a
locally compact quantum group to its intrinsic group is a functor between
the appropriate categories, show that we can identify the intrinsic group as
the ``maximal'' classical subgroup, and then use this to show that the
``intrinsic functor'' is the left adjoint to the inclusion functor from locally
compact groups to locally compact quantum groups.  In fact, we shall show that
the intrinsic group is a \emph{closed} subgroup, in the sense of \cite{dkss},
in fact, in the strong \emph{Vaes closed} sense.

There are a number of different, equivalent, definitions of the intrinsic group,
and these different definitions have interesting interactions with the different
presentations of a morphism between quantum groups.  We wish to be rather careful
about the isomorphisms involved, and furthermore, we also want to consider the
interaction with $C_0^u(\G)$.  Thus, we shall expound some of the results from
\cite{kn} in detail.

The following is the key technical lemma; two rather different proofs can be
found in \cite[Theorem~3.2]{dp} and \cite[Theorem~3.9]{kn}.  

\begin{proposition}
Let $\G$ be a locally compact quantum group, and let $x\in L^\infty(\G)$ be
non-zero with $\Delta(x) = x\otimes x$.  Then $x$ is unitary, and
$x\in M(C_0(\G))$.
\end{proposition}

That is, all characters on $L^1(\G)$ arise from one-dimensional unitary
corepresentations of $\G$.

\begin{lemma}
Let $x\in M(C_0^u(\G))$ with $\Delta(x) = x\otimes x$.  Then $x$ is unitary.
\end{lemma}
\begin{proof}
Let $y = \pi_\G(x) \in M(C_0(\G))$ so $\Delta(y) = y\otimes y$ and hence $y$ is
unitary.  Then we use that
\[ x\otimes y = (\id\otimes\pi_\G)\Delta(x) = \Ww^*(1\otimes y) \Ww \]
and so $x\otimes 1 = (1\otimes y^*)\Ww^*(1\otimes y) \Ww \in M(C_0^u(\G)
\otimes\mc B_0(L^2(\G)))$ is unitary, and so $x$ is unitary.
\end{proof}

The following then states the different, equivalent, definitions of the
intrinsic group, compare \cite[Theorem~3.12]{kn}.  We claim that the following
sets, given the stated topologies, are locally compact groups, and are all
homeomorphic (for maps to be defined shortly):
\begin{enumerate}
\item\label{def:ig:one}
The collection of completely positive, completely isometric isomorphisms
in $C_{cb}^l(L^1(\G))$, with composition as the group product, and the strong
operator topology, denoted by $\tilde\G$;
\item\label{def:ig:two}
The spectrum of the $C^*$-algebra $C_0^u(\G)$, that is, the collection of
non-zero characters on $C_0^u(\G)$ with the relative weak$^*$-topology,
and the product induced by $\Delta$, denoted by $\spec(C_0^u(\G))$;
\item\label{def:ig:three}
The intrinsic group of $L^\infty(\hh\G)$, namely $\Gr(\hh\G) =
\{ \hh u\in L^\infty(\G) : \hh\Delta(\hh u)=\hh u\otimes \hh u, \hh u\not=0 \}$,
with the product from $L^\infty(\hh\G)$, and the relative weak$^*$-topology;
\item\label{def:ig:four}
The ``universal intrinsic group'' of $C_0^u(\G)$, namely
$\Gr_u(\hh\G) = \{ \hh u\in M(C_0^u(\hh\G)) : \hh\Delta(\hh u)=\hh u\otimes
\hh u, \hh u\not=0 \}$ with the product from $M(C_0^u(\hh\G))$, and the
relative strict topology.
\end{enumerate}

Let us note that the ``intrinsic group'' is often defined by requiring that
$\hh u$ be invertible; but by our technical lemma, $\hh u$ is automatically
unitary.  We note that (\ref{def:ig:four}) is a new equivalence not previously
studied.

We now define the maps between these sets.  We choose slightly different
conventions to \cite{kn}, in particular, swap $\hh u$ for $\hh u^*$, as our
conventions seem more natural given the later results.  Given $\hh u\in
\Gr(\hh\G)$, we identify $M(C_0(\hh\G))$ with $M(\mathbb C \otimes C_0(\hh\G))$
and then observe that $(\id\otimes\hh\Delta)(\hh u) = \hh u_{13}\hh u_{12}$.
So by \cite[Proposition~5.3]{kusuni} there
is a non-degenerate $*$-homomorphism $\gamma:C_0^u(\G) \rightarrow \mathbb C$,
that is, $\gamma\in\spec(C_0^u(\G))$, with $u = (\gamma\otimes\id)(\Ww)$.
It is easy to see that this in fact gives a bijection between
$\spec(C_0^u(\G))$ and $\Gr(\hh\G)$.

By pushing things to the universal level, and using $\WW$ and (the dual version
of) \cite[Proposition~6.5]{kusuni}, we also get a bijection between
$\spec(C_0^u(\G))$ and $\Gr_u(\hh\G)$ which identifies $\gamma$ with
$(\gamma\otimes\id)(\WW)$.  Then the strict extension of $\pi:C_0^u(\G)
\rightarrow C_0(\G)$ restricts to a bijection between $\Gr(\hh\G)$ and
$\Gr_u(\hh\G)$.

Finally, the bijection between $\tilde\G$ and $\Gr(\hh\G)$ follows from
\cite[Theorem~4.7]{jnr}, compare \cite[Theorem~3.7]{kn}.  In the remainder
of this section, we shall give an alternative proof, using \cite{daws1}, and
also give a concise proof that the maps constructed are homeomorphisms;
of course, these results are new for $\Gr_u(\hh\G)$.

In the next section, we start to study how morphisms and intrinsic groups
interact.  The following, which is really not made explicit in \cite{kn},
will be vital for that purpose.

\begin{theorem}\label{thm:in_to_mult}
The bijection between $C_{cb}^l(L^1(\G))$ and $M_{cb}^l(L^1(\G)) \subseteq
M(C_0(\hh\G))$ restricts to a bijection between $\tilde\G$ and $\Gr(\hh\G)$.
Furthermore, if $L_*$ and $\hh u$ are thus associated, then $L(x) = \hh u^* x
\hh u$ for $x\in L^\infty(\G)$.
\end{theorem}
\begin{proof}
We use the main result of \cite{daws1}, which tells us that there is a natural
bijection between \emph{completely positive} multipliers of $L^1(\G)$ and
$C_0^u(\G)^*_+$.  Indeed, for such $L_*$ there is $\mu\in C_0^u(\G)^*$ positive
such that, embedding $L^1(\G)$ into $C_0^u(\G)^*$, we have that $L_*$ is given
by left multiplication by $\mu$.  That is,
\[ \omega\circ L\circ\pi = (\mu\otimes\omega\circ\pi)\Delta
= (\mu\otimes\omega)(\Ww^*(1\otimes\pi(\cdot))\Ww). \]

If $L_* \in \tilde\G$ then $L_*$ is a completely isometric isomorphism,
so there exists a completely isometric $L_*^{-1}$.  Thus $L_*^{-1}$ is also
completely positive, and it is easy to see that $L_*^{-1}$ is also a left
multiplier, compare the proof of \cite[Theorem~4.7]{jnr}.  So choose $\mu^{-1}
\in C_0^u(\G)^*$ for $L_*^{-1}$.  Both $L$ and $L^{-1}$ must be unital,
completely positive, and so $\mu,\mu^{-1}$ are states.

For $x\in L^\infty(\G), \omega\in L^1(\G)$,
\begin{align*} \ip{x}{\omega}
&= \ip{L^{-1}(L(x))}{\omega}
= \ip{\mu^{-1}\otimes\omega}{\Ww^*(1\otimes L(x))\Ww} \\
&= \ip{\mu\otimes\mu^{-1}\otimes\omega}{\Ww^*_{23}\Ww^*_{13}(1\otimes 1\otimes x)
   \Ww_{13}\Ww_{23}} \\
&= \ip{(\mu^{-1}\otimes\mu)\Delta\otimes\omega}{\Ww^*(1\otimes x)\Ww},
\end{align*}
as $\Ww$ is a corepresentation of $C_0^u(\G)$.  Apply this to
$x = (\id\otimes\hh\omega)(W)$ to see that
\[ x = (\mu^{-1}\star\mu\otimes\id)(\Ww^*(1\otimes x)\Ww)
= (\mu^{-1}\star\mu\otimes\id\otimes\hh\omega)(\Ww^*_{12}W_{23}\Ww_{12}). \]
As $\Ww^*_{12}W_{23}\Ww_{12} = ((\id\otimes\pi)\Delta\otimes\id)(\Ww)
= \Ww_{13} W_{23}$ it follows that
\[ (\id\otimes\hh\omega)(W) =
(\mu^{-1}\star\mu\otimes\id\otimes\hh\omega)(\Ww_{13} W_{23}), \]
and so $W = (\mu^{-1}\star\mu\otimes\id\otimes\id)(\Ww_{13} W_{23})$
and so $(\mu^{-1}\star\mu\otimes\id)(\Ww) = 1$.  Similarly
$(\mu\star\mu^{-1}\otimes\id)(\Ww) = 1$.  By \cite[Proposition~6.3]{kusuni}
and its proof, it follows that $\mu^{-1}\star\mu = \mu\star\mu^{-1} = \epsilon$
the counit of $C_0^u(\G)$.

Let $T = (\mu\otimes\id)\Delta : C_0^u(\G)\rightarrow C_0^u(\G)$, a unital
completely positive map.  That $T$ maps into $C_0^u(\G)$, and not $M(C_0^u(\G))$,
follows by observing that $\{ (\id\otimes\omega)(\Ww) : \omega\in L^1(\hh\G) \}$
is norm dense in $C_0^u(\G)$, see the discussion after
\cite[Proposition~5.1]{kusuni}, and then calculating that
\[ T\big( (\id\otimes\omega)(\Ww) \big)
= (\mu\otimes\id\otimes\omega)(\Ww_{13} \Ww_{23})
= (\id\otimes\omega a)(\Ww) \in C_0^u(\G), \]
where $a = (\mu\otimes\id)(\Ww) \in L^\infty(\hh\G)$.  We similarly form $T^{-1}$,
and observe that $T^{-1}$ is the inverse of $T$.  Indeed,
\begin{align*} T^{-1}(T(x)) &=
(\mu^{-1}\otimes\id)\Delta\big( (\mu\otimes\id)\Delta(x) \big)
= (\mu^{-1}\otimes\mu\otimes\id)\Delta^2(x) \\
&= ((\mu^{-1}\otimes\mu)\Delta\otimes\id)\Delta(x)
= (\epsilon\otimes\id)\Delta(x) = x. \end{align*}

We now use the Schwarz inequality, and the theory of multiplicative domains,
for completely positive maps, see \cite[Proposition~1.5.7]{bo} for example.
For $a\in C_0^u(\G)$,
\[ a^*a = T^{-1}(T(a))^* T^{-1}(T(a))
\leq T^{-1}(T(a)^*T(a)) \leq T^{-1}(T(a^*a)) = a^*a, \]
and so we have equality throughout, namely $a^*a = T^{-1}(T(a)^*T(a))$
or equivalently, $T(a^*a) = T(a)^* T(a)$.  Similarly we can show that
$T(aa^*) = T(a) T(a)^*$, and it hence follows that $T(ab) = T(a)T(b)$ for all
$a,b\in C_0^u(\G)$.  Thus $T$ is a $*$-automorphism of $C_0^u(\G)$.
As $\mu = \epsilon\circ T$ it follows that $\mu$ is a character.

Then let $\hh u = (\mu\otimes\id)(\Ww)$ so $\hh\Delta(\hh u)
= (\mu\otimes\id\otimes\id)(\Ww_{13} \Ww_{12})
= \hh u\otimes\hh u$ as $\mu$ is multiplicative, so $\hh u\in \Gr(\hh\G)$.
Then, for $x\in L^\infty(\G)$,
\[ L(x) = (\mu\otimes\id)(\Ww^*(1\otimes x)\Ww)
= \hh u^* x \hh u, \]
as claimed.  Finally, that $\hh\Delta(\hh u) = \hh u\otimes \hh u$ is
equivalent to $\hh W^*(1\otimes \hh u) \hh W = \hh u\otimes \hh u$ or
equivalently that $W(\hh u\otimes 1)W^* = \hh u\otimes \hh u$, and so
\[ (L\otimes\id)(W) = (\hh u^*\otimes 1)W(\hh u\otimes 1) 
= (\hh u^*\otimes 1)(\hh u\otimes \hh u) W
= (1\otimes\hh u)W, \]
and so $\hh u\in M_{cb}^l(L^1(\G))$ is associated to $L$ as required.
\end{proof}

Notice that the previous proof did not use weight theory (and that neither
does \cite{daws1}).  We now show that our maps are homeomorphisms: this is
a new result for the equivalence with $\Gr_u(\hh\G)$, and for completeness,
we give a complete proof.  As $\tilde\G$ is easily seen to be a topological
group, \cite[Proposition~3.5]{kn}, and $\spec(C_0^u(\G))$ is locally compact,
it follows that the intrinsic group is indeed a locally compact group.
Our proof will avoid use of weight theory, standard position of von Neumann
algebras etc., compare the proof of \cite[Theorem~3.7]{kn}.

\begin{theorem}
The bijections between our four equivalent conditions are homeomorphisms.
\end{theorem}
\begin{proof}
Firstly, the map $\spec(C_0^u(\G))\rightarrow\Gr(\hh\G); \gamma\mapsto
(\gamma\otimes\id)(\Ww)$ is a homeomorphism.  This follows, as $\gamma_i
\rightarrow\gamma$ is equivalent to $\lim_i \gamma_i(a) = \gamma(a)$ for
all $a\in C_0^u(\G)$ of the form $(\id\otimes\hh\omega)(\Ww)$ for $\hh\omega
\in L^1(\hh\G)$, as such $a$ are norm dense, and $(\gamma_i)$ is a bounded net.
However, this is clearly equivalent to $(\gamma_i\otimes\id)(\Ww)
\rightarrow (\gamma\otimes\id)(\Ww)$ weak$^*$ in $L^\infty(\hh\G)$, as required.

We next show that $\tilde\G \rightarrow \Gr_u(\hh\G)$ is continuous.  
Let the (bounded) net $(L_{i,*}) \subseteq \tilde\G\subseteq C_{cb}^l(L^1(\G))$
converge strongly to $L_*$, and be associated to $\hh u_i\in\Gr_u(\hh\G)$, with
$L_*$ associated to $\hh u$.  As $(\hh u_i)$ is a net of unitaries, to show that
$\hh u_i a \rightarrow \hh u a$ for each $a\in C_0^u(\hh\G)$, it suffices to
check for a dense collection of such $a$, for example,
$a = (\omega\otimes\id)(\wW)$ for $\omega\in L^1(\G)$.  However,
\begin{align*} \lim_i \hh u_i (\omega\otimes\id)(\wW)
&= \lim_i (\omega\otimes\id)( (1\otimes\hh u_i)\wW )
= \lim_i (\omega\otimes\id)(L_i\otimes\id)(\wW) \\
&= \lim_i (L_{i,*}(\omega)\otimes\id)(\wW)
= (L_{*}(\omega)\otimes\id)(\wW)
= \hh u (\omega\otimes\id)(\wW),
\end{align*}
as required.  We similarly need to show that $\hh u_i^* (\omega\otimes\id)(\wW)
\rightarrow \hh u^* (\omega\otimes\id)(\wW)$ for each $\omega$.
However, as $\hh u_i^* = \hh u_i^{-1}$, this claim will follow because
$L_{i,*}^{-1} \rightarrow L_*^{-1}$.  Thus $\hh u_i\rightarrow \hh u$ strictly,
as claimed.

That $\Gr_u(\hh\G) \rightarrow \Gr(\hh\G)$ is continuous follows easily, as
$\pi$ is strictly continuous, and strict convergence in $M(C_0(\hh\G))$ implies
weak$^*$-convergence in $L^\infty(\hh\G)$.

Finally we show that $\Gr(\hh\G) \rightarrow \tilde\G$ is continuous.
Continuing with the same notation, suppose that $\hh u_i\rightarrow\hh u$
weak$^*$ in $\Gr(\hh\G)$.
As each $\hh u_i$ is unitary, this implies that actually $\hh u_i
\rightarrow \hh u$ strongly, in $\mc B(L^2(\G))$.  This then implies that for
all $\xi,\eta\in L^2(\G)$,
\[ \lim_i \hh u_i\omega_{\xi,\eta} \hh u_i^* =
\lim_i \omega_{\hh u_i\xi,\hh u_i\eta}
= \omega_{\hh u\xi,\hh u\eta} = \hh u\omega_{\xi,\eta} \hh u^* \]
in $\mc B(L^2(\G))_*$ and hence also in $L^1(\G)$.  However, as
$L_i(x) = \hh u_i^* x \hh u_i$, this shows that $L_{i,*}(\omega_{\xi,\eta})
\rightarrow L_*(\omega_{\xi,\eta})$ in $L^1(\G)$.  As $(L_{i,*})$ is a
bounded net, it follows that $L_{i,*}\rightarrow L_*$ strongly, as required.
\end{proof}

\section{The Intrinsic Group functor}\label{sec:igf}

In this section, we shall show that the assignment $\G \rightarrow \tilde\G$
is actually a functor.  Given the results of Section~\ref{sec:mult_mor}, we
have little choice as to how a morphism $\G\rightarrow\H$ should map
$\tilde\G$ to $\tilde\H$, as $\tilde\G$ is realised as a subset of the
multipliers of $L^1(\G)$.  Fortunately, this works!

\begin{theorem}
Let $f:\G\rightarrow\H$ be a morphism of quantum groups, which induces the
completely contractive homomorphism $C_{cb}^l(L^1(\G)) \rightarrow
C_{cb}^l(L^1(\H))$, as before.  This restricts to a continuous group
homomorphism $\tilde f:\tilde\G \rightarrow \tilde \H$.  The assignment
$f\mapsto\tilde f$ is a functor.

Indeed, let the morphism be represented by
$\phi:C_0^u(\H)\rightarrow M(C_0^u(\G))$ and
$\beta:C_0(\H)\rightarrow M(C_0(\G)\otimes C_0(\H))$, with dual
counterparts $\hh\phi$ and $\hh\beta$.
Let $L_*\in \tilde\G$ be associated to $\gamma\in\spec(C_0^u(\G)),
\hh u\in\Gr(\hh\G)$ and $\hh u_u \in \Gr_u(\hh\G)$, and let $L_*$ be mapped
to $L_*' \in \tilde\H$, associated to $\gamma', \hh u', \hh u_u'$.  Then
we have the following relations:
\begin{itemize}
\item $\hh u_u' = \hh\phi(\hh u_u)$;
\item $\hh\beta(\hh u) = \hh u' \otimes \hh u$;
\item $\beta L' = (L\otimes\id)\beta$;
\item $\gamma' = \gamma\circ\phi$.
\end{itemize}
\end{theorem}
\begin{proof}
By definition, the map $C_{cb}^l(L^1(\G)) \rightarrow
C_{cb}^l(L^1(\H))$ is induced by the restriction of $\hh\phi$ to a map
$M_{cb}^{l,u}(L^1(\G)) \rightarrow M_{cb}^{l,u}(L^1(\H))$, see
Theorem~\ref{thm:main_mor}.  As $\hh\phi$ is a Hopf $*$-homomorphism,
it's clear that $\hh\phi(\hh u_u) \in \Gr_u(\hh\H)$ for each $\hh u_u \in
\Gr_u(\hh\G)$.  So we do obtain a map $\tilde f:\tilde\G \rightarrow \tilde \H$.
As the product on $\Gr_u(\hh\G)$ is simply the restriction of the product
on $M(C_0^u(\G))$, and as $\hh\phi$ is a homomorphism, the map
$\tilde f:\tilde\G \rightarrow \tilde \H$ is a group homomorphism, clearly
continuous.  Finally, because composition of morphisms is given by
composition of the associated Hopf $*$-homomorphisms, it is clear that
$f\mapsto\tilde f$ is a functor, namely that if $h=g\circ f$ then
$\tilde h = \tilde g \circ \tilde f$.

By Lemma~\ref{lem:cent_bichar} and Theorem~\ref{thm:in_to_mult},
we then see that
\[ (1\otimes \hh u')U = (L\otimes\id)(U) = (\hh u^*\otimes 1)U(\hh u\otimes 1) \]
and so, as $\hh U = \sigma(U^*)$, it follows that
\[ (\hh u'\otimes 1)\hh U^* = (1\otimes \hh u^*)\hh U^*(1\otimes\hh u), \]
which in turn implies that
\[ \hh\beta(\hh u) = \hh U^*(1\otimes\hh u)\hh U
= \hh u'\otimes\hh u, \]
as claimed.

That $\beta L' = (L\otimes\id)\beta$ follows immediately from
Corollary~\ref{corr:for_coactions}, and that $\gamma' = \gamma\circ\phi$
follows immediately from Proposition~\ref{prop:comm_diag}.
\end{proof}

\subsection{Universal property}\label{sec:uni_prop}

In this section, we shall construct a morphism $\tilde\G \rightarrow \G$,
and show that $\tilde\G$ satisfies a natural universal property.  We then
draw some category theoretic conclusions.  In the next section, we shall show
that actually $\tilde\G$ is a ``closed quantum subgroup'' of $\G$.

Let us view $\tilde\G$ as being $\Gr_u(\hh\G)$ with the strict topology.
Thus the formal identity map $\tilde\G \rightarrow M(C_0^u(\hh\G))$ is a
continuous homomorphism, when $M(C_0^u(\hh\G))$ carries the strict topology.
If $C_0^u(\hh\G)\subseteq\mc B(H)$ is a faithful non-degenerate
$*$-representation, then $M(C_0^u(\G))\subseteq\mc B(H)$ as well, and the
induced map $\tilde\G \rightarrow\mc B(H)$ will be a strongly continuous
unitary representation.  By the universal property of $C^*(\tilde\G)$,
we hence obtain a $*$-homomorphism $C^*(\tilde\G) \rightarrow \mc B(H)$; and
one can show that this takes values in $M(C_0^u(\hh\G))$.  So we obtain
a non-degenerate $*$-homomorphism $\hh\theta:C^*(\tilde\G) \rightarrow
M(C_0^u(\hh\G))$.  The strict extension of this map sends an element of
$\tilde\G$ to its image in $\Gr_u(\hh\G)$, and thus by the definition of the
coproduct on $C^*(\tilde\G)$, we see that $\hh\theta$ is a Hopf
$*$-homomorphism.  Hence we have a morphism $\hh\G \rightarrow \hh{\tilde\G}$
and so by duality, the claimed morphism $\tilde\G \rightarrow \G$.

We constructed this morphism from what might be called a group
representation perspective.  The following shows that it also has an extremely
natural interpretation at the $C^*$-algebra level.

\begin{proposition}\label{prop:gelfand}
Let the morphism $\tilde\G \rightarrow \G$ induce the Hopf $*$-homomorphism
$\theta: C_0^u(\G) \rightarrow M(C_0(\tilde\G))$.  Then $\theta$ maps into
$C_0(\tilde\G)$, and is surjective.  Viewing $\tilde\G$ as $\spec(C_0^u(\G))$,
the map $\theta$ is nothing but the Gelfand transform.
\end{proposition}
\begin{proof}
For a good treatment of Gelfand transforms for non-unital algebras, see
\cite[Section~2.3]{dales}, in particular \cite[Theorem~2.3.25]{dales}.
In keeping with our notation, let $C_0^u(\G) \rightarrow C_0(\tilde\G);
a\mapsto \tilde a$ be the Gelfand transform, so that $\tilde a(\gamma)
= \ip{\gamma}{a}$ for $a\in C_0^u(\G), \gamma \in \tilde \G = \spec(C_0^u(\G))$.
Then the algebra $\{ \tilde a : a\in C_0^u(\G) \}$ is self-adjoint, separates
the points of $\tilde\G$, and separates the points from $0$, and hence is
dense in $C_0(\tilde\G)$.  We conclude that the Gelfand transform is onto.

It hence remains to show that $\theta(a) = \tilde a$ for each $a\in C_0^u(\G)$.
Consider the universal bicharacter for $\tilde\G$.  This is
\[ \WW_{\tilde\G} = \wW_{\tilde\G} \in M(C_0(\tilde\G) \otimes C^*(\tilde\G))
= C_b^{str}(\tilde\G, M(C^*(\tilde\G))), \]
the space of bounded strictly continuous maps $\tilde\G \rightarrow
M(C^*(\tilde\G))$, and under this identification, $\WW_{\tilde\G}$ is nothing but
the inclusion $\tilde\G \rightarrow M(C^*(\tilde\G))$.  By definition,
\[ (\theta\otimes\id)(\WW_{\G}) = (\id\otimes\hh\theta)(\WW_{\tilde\G})
\in C_b^{str}(\tilde\G, M(C_0^u(\hh\G))), \]
and by the construction of $\hh\theta$, this is the inclusion
$\tilde\G = \Gr_u(\hh\G) \rightarrow M(C_0^u(\hh\G))$.
Let $\gamma\in\spec(C_0^u(\G))$ be associated to $\hh u_u \in \Gr_u(\hh\G)$.
Viewing $(\theta\otimes\id)(\WW_\G)$ as a strictly continuous function
$\tilde\G \rightarrow M(C_0^u(\hh\G))$, the value of this function at $\gamma$
is hence simply $\hh u_u$.  However,
this is equal to $(\gamma\otimes\id)(\WW_\G)$.

So, we have that $(\theta\otimes\id)(\WW_\G) \in M(C_0(\tilde\G)\otimes
C_0^u(\hh\G)) = C_b^{str}(\tilde\G, M(C_0^u(\hh\G)))$ is the function
$\gamma \mapsto (\gamma\otimes\id)(\WW_\G)$.  Apply $\id\otimes\pi_{\hh\G}$
and then slice by some $\hh\omega\in L^1(\hh\G)$ to see that
$\theta( (\id\otimes\hh\omega)(\Ww) ) \in M(C_0(\tilde\G))$ is the function
\[ \gamma \mapsto \ip{\gamma}{(\id\otimes\hh\omega)(\Ww)}. \]
As the collection of elements $(\id\otimes\hh\omega)(\Ww)$ is dense in
$C_0^u(\G)$, we conclude that $\theta$ is indeed nothing but the Gelfand transform.
\end{proof}

That $\theta$ is a surjection $C_0^u(\hh\G)\rightarrow C_0(\tilde\G)$ means
that $\tilde\G$ is identified as a \emph{Woronowicz closed quantum subgroup}
of $\G$, see \cite{dkss}.  In the next section, we shall prove that $\tilde\G$
satisfies the a priori stronger condition of being \emph{Vaes closed}.

We also note that if the reader is unhappy with the slightly sketchy proof just
given, then we could simply \emph{define} $\theta$ to be the Gelfand transform:
it is very easy to show that $\theta$ is then a Hopf $*$-homomorphism.  However,
we feel that for motivational purposes, the definition of $\hh\theta$ is more
natural.

The following shows that $\tilde\G$ is then the maximal ``classical''
subgroup of $\G$.

\begin{theorem}\label{thm:uni_prop}
The morphism $\tilde\G \rightarrow\G$ satisfies the following universal
property.  For any locally compact group $H$, and any morphism
$H \rightarrow \G$, there is a unique continuous group homomorphism
$H\rightarrow\tilde\G$ making the following, equivalent, diagrams commute:
\[ \xymatrix{ H \ar[r] \ar[d]_{\exists\,!} & \G \\
\tilde\G \ar[ru] } \qquad
\xymatrix{ M(C_0(H)) & C_0^u(\G) \ar[l] \ar[ld]_{\theta} \\
C_0(\tilde\G) \ar[u] } \]
\end{theorem}
\begin{proof}
That the two diagrams are equivalent follows from the definition of
what a morphism of locally compact quantum groups is.  As $\theta$ is
onto, if a Hopf $*$-homomorphism $C_0(\tilde\G) \rightarrow M(C_0(H))$
exists, making the diagram commute, then it is uniquely defined.

Let $\phi:C_0^u(\G)\rightarrow M(C_0(H))$ be our Hopf $*$-homomorphism.  For
each $h\in H$ let $\delta_h:C_0(H)\rightarrow\mathbb C$ be the functional
given by point evaluation.  So $\delta_h$ is a character, and as $\phi$
is non-degenerate, $\delta_h\circ\phi$ is a character, and hence defines
a member of $\tilde\G$.  We hence obtain our map $H\rightarrow\tilde\G$, which
is easily seen to be continuous, and thus induces
$\psi : C_0(\tilde\G) \rightarrow M(C_0(H))$.  Then, for
$a\in C_0^u(\G)$ and $h\in H$,
\[ \psi(\theta(a))(h) = \theta(a)(\delta_h\circ\phi)
= \ip{\delta_h\circ\phi}{a} = \phi(a)(h), \]
and so $\psi\circ\theta = \phi$ as required.  As $\theta,\phi$ are Hopf
$*$-homomorphisms,
\[ \Delta_H \psi \theta = \Delta_H \phi = (\phi\otimes\phi)\Delta_\G
= (\psi\theta\otimes\psi\theta)\Delta_\G
= (\psi\otimes\psi)\Delta_{\tilde\G}\theta. \]
As $\theta$ is onto, it follows that $\psi$ is a Hopf $*$-homomorphism,
and so our map $H\rightarrow\tilde\G$ is a group homomorphism, as required.
\end{proof}

We can of course dualise this universal property, and obtain two more
commuting diagrams
\[ \xymatrix{ \hh H & \hh\G \ar[l] \ar[ld] \\
\hh{\tilde\G} \ar[u]^{\exists\,!} } 
\qquad
\xymatrix{ C^*(H) \ar[r]^{\hh\phi} \ar[d] & M(C_0^u(\hh\G)) \\
C^*(\tilde\G) \ar[ru]_{\hh\theta} } \]
As $\hh\phi:C^*(H)\rightarrow M(C_0^u(\hh\G))$ is a Hopf $*$-homomorphism, 
the strict extension must send $h\in H\subseteq M(C^*(H))$ to a member of
$\Gr_u(\hh\G) = \tilde\G$, and this gives the map $C^*(H)\rightarrow
C^*(\tilde\G)$.

Let us think about this result from the perspective of some elementary
category theory.  Let ${\sf LCG}$ and ${\sf LCQG}$ be the categories of
locally compact groups and locally compact quantum groups, respectively.
The definition of a morphism in ${\sf LCQG}$ is setup precisely so that the
assignment of $G\in{\sf LCG}$ to $(C_0(G),\Delta_G)\in{\sf LCQG}$ \emph{is}
a functor.  Let this functor be $\mc C:{\sf LCG} \rightarrow {\sf LCQG}$,
we choose $\mc C$ for ``classical''.  Let $\mc I:{\sf LCQG}\rightarrow
{\sf LCG}$ be the ``intrinsic group functor'', the assignment of $\tilde\G$
to $\G$.  We recall the notion of an adjoint functor, see for example
\cite[Chapter~2]{leinster}.

\begin{theorem}
The functor $\mc I$ is a right adjoint to the functor $\mc C$.
\end{theorem}
\begin{proof}
This is equivalent to $\mc C$ being a left adjoint to $\mc I$.
There are a number of different, equivalent meanings to this.  One is that
$\mc C:{\sf LCQG}\leftarrow{\sf LCG}$ is a left adjoint functor if for each
$\G\in{\sf LCQG}$ there is a \emph{terminal morphism} from $\mc C$ to $\G$.
That is, there exists $\tilde\G\in{\sf LCG}$ and $\mc C(\tilde\G)
\rightarrow\G$ such that for each $H\in{\sf LCG}$ and each morphism
$f:\mc C(H)\rightarrow\G$, there is a unique morphism $H\rightarrow\tilde\G$ with
\[ \xymatrix{ H \ar[d]_{\exists\,! g} & \mc C(H) \ar[d]_{\mc C(g)}
\ar[rd]^f \\
\tilde\G & \mc C(\tilde\G) \ar[r] & \G } \]
However, if we remember that $\mc C$ is essentially the ``inclusion'', then
this is nothing but the universal property of $\tilde\G$ just shown in
Theorem~\ref{thm:uni_prop}.  In this situation, it is then actually automatic
(from purely category theoretic considerations) that $\mc I:\G\rightarrow\tilde\G$
is a functor (the universal property alone can be used to construct
$\mc I(f):\tilde\G\rightarrow\tilde\H$ given any $f:\G\rightarrow\H)$.
\end{proof}

\subsection{Closed subgroups}\label{sec:closed}

In \cite{dkss} the notion of a \emph{closed quantum subgroup} was explored.
We say that a morphism $\H\rightarrow\G$ identifies $\H$ as a closed quantum
subgroup of $\G$, in the sense of Woronowicz, if the Hopf $*$-homomorphism
$C_0^u(\G)\rightarrow M(C_0^u(\H))$ maps into, and onto, $C_0^u(\H)$.
This notion is easily seen to reduce to the classical notion of a closed
group, when applied to a classical group $G$, see \cite[Sections~3, 4]{dkss}.

When $H\subseteq G$ is a closed subgroup of a locally compact group, the
Herz Restriction Theorem (see \cite{herz})
tells us that the restriction map gives a quotient
map from the Fourier algebra $A(G)$ to $A(H)$, or equivalently, gives a
normal injective $*$-homomorphism $VN(H)\rightarrow VN(G)$.  This notion
was generalised to quantum groups in \cite[Definition~2.5]{vaesii}, and
motivated the authors of \cite{dkss} to give the notion of a Vaes closed
quantum subgroup.  Here we state the original definition: $\H\rightarrow\G$
identifies $\H$ as a Vaes closed quantum subgroup of $\G$ if there is a
normal, unital, injective $*$-homomorphism $\gamma : L^\infty(\hh\H)\rightarrow
L^\infty(\hh\G)$ with $\pi_{\hh\G} \circ \hh\phi = \gamma \circ \pi_{\hh\H}$
where $\hh\phi:C_0^u(\hh\H) \rightarrow M(C_0^u(\hh\G))$ is the Hopf
$*$-homomorphism representing the dual morphism $\hh\G\rightarrow\hh\H$.

We observed above that $\tilde\G$ is Woronowicz closed in $\G$.  
In \cite[Theorem~3.14]{kn} it's shown that $\G\rightarrow\tilde\G$ preserves
compactness and discreteness.  As $\tilde\G$ is a closed quantum subgroup of
$\G$, this result now also follows immediately from results of \cite{dkss}.
Indeed, if $\G$ is compact, then $C_0^u(\G)$ is unital, so $C_0(\tilde\G)$ is
unital, so $\tilde\G$ is compact.  The discrete case is more involved, see
\cite[Theorem~6.2]{dkss}.

We now prove a stronger result, that $\tilde\G$ is Vaes closed.

\begin{theorem}
The morphism $\theta:C_0^u(\hh\G) \rightarrow C_0(\tilde\G)$ identifies
$\tilde\G$ as a Vaes closed quantum subgroup of $\G$.
\end{theorem}
\begin{proof}
Consider the dual morphism $\hh\theta:C^*(\tilde\G) \rightarrow
M(C_0^u(\hh\G))$.  Compose with $\pi_{\hh\G}$ and let $M$ be the von Neumann
algebra generated by the image in $L^\infty(\hh\G)$.  As $\pi_{\hh\G}\circ
\hh\theta$ is a Hopf $*$-homomorphism, it follows by weak$^*$-continuity
that $\Delta_{\hh\G}(x) \in M \vnten M \subseteq L^\infty(\hh\G)\vnten
L^\infty(\hh\G)$ for each $x\in M$.  By \cite[Remark~12.1]{kusuni}
we also know that $\hh\theta$ intertwines the scaling group $(\tau_t)$
and the unitary antipode $R$; the same is true of $\pi_{\hh\G}$.
It follows that $(\tau_t)$ restricts to the identity map on $M$, and that
$R$ restricts to $M$.  It follows from \cite[Proposition~A5]{BV}, that $M$
``is a quantum group'', that is, admits invariant weights.  For us, an
convenient way to restate this is that there is a locally compact quantum
group $\K$ and a normal $*$-isomorphism $\psi : L^\infty(\K)\rightarrow M
\subseteq L^\infty(\hh\G)$
with $\psi$ intertwining the coproducts of $\K$ and $\hh\G$.  As
$C^*(\tilde\G)$ is cocommutative and its image is weak$^*$-dense in
$L^\infty(\K)$, it follows that $\K$ is cocommutative, and so there is a
locally compact group $K$ with $L^\infty(\K) = VN(K)$.

Thus, we obtain the following commutative diagram:
\[ \xymatrix{ C^*(\tilde\G) \ar[r]^{\hh\theta} \ar[rd] &
M(C_0^u(\hh\G)) \ar[r]^\pi &
M(C_0(\hh\G)) \ar[r] &
L^\infty(\hh\G) \\
& L^\infty(\K) = VN(K) \ar[rru]_{\psi} } \]
where $\psi$ is a normal $*$-isomorphism onto its range, which intertwines
the coproducts.  This means, in particular, that $\psi$ restricts to a map
$K = \Gr(VN(K)) \rightarrow \Gr(\hh\G) = \tilde\G$, and so we obtain an
injective group homomorphism $f:K \rightarrow \tilde\G$.  As $\psi$ is normal,
by the definition of the topologies on $\Gr(VN(K))$ and $\Gr(\hh\G)$, we see
that $f$ is continuous.

It follows from the discussion in Section~\ref{sec:uni_prop} that, if we
regard $\tilde\G$ as being the set $\Gr(\hh\G)$, then the map
$\pi\hh\theta$ sends $\hh u \in \tilde \G \subseteq M(C^*(\tilde\G))$ to
$\hh u \in \Gr(\hh\G) \subseteq L^\infty(\hh\G)$.  As such $\hh u\in M$,
we see that $\psi^{-1}(\hh u) \in \Gr(VN(K))$, and so we obtain again a
continuous group homomorphism  $g:\tilde\G \rightarrow K$.  Then $f\circ g:
\tilde\G \rightarrow \tilde\G$ is the identity, and as $f$ is injective,
$f$ and $g$ are mutual inverses.

Thus $K\cong\tilde\G$ and so $VN(K) \cong VN(\tilde\G)$.  The isomorphisms
involved ensure that the induced map $\theta_0:VN(\tilde\G) \rightarrow
L^\infty(\hh\G)$ restricts to the identity map on $\tilde\G = \Gr(\hh\G)$.
Hence we obtain the following commutative diagram:
\[ \xymatrix{ VN(\tilde\G) \ar[d]_{\cong} \ar[rd]^{\theta_0}
& C^*(\tilde\G) \ar[l]_{\pi} \ar[d]^{\pi\hh\theta} \\
VN(K) \ar[r]^{\psi} & L^\infty(\hh\G) } \]
as required.
\end{proof}

We have hence obtained a normal injective unital $*$-homomorphism
$\theta_0:VN(\tilde\G) \rightarrow L^\infty(\hh\G)$ such $\pi_{\hh\G}
\hh\theta = \theta_0\pi_{\hh{\tilde\G}}$.  By the definition of $\hh\theta$,
it follows that if we identify $\tilde\G$ with $\Gr(\hh\G)$ then
$\theta_0$ is simply the extension of the ``inclusion'' $VN(\tilde\G)
\supseteq \tilde \G = \Gr(\hh\G) \rightarrow L^\infty(\hh\G)$.  In particular,
the von Neumann algebra generated by $\Gr(\hh\G)$ is isomorphic to
$VN(\tilde\G)$.

The pre-adjoint of $\theta_0$ gives us a (complete) quotient map $L^1(\hh\G)
\rightarrow A(\tilde\G)$, the Fourier algebra of $\tilde\G$.
That is, for each $\hh\omega\in L^1(\hh\G)$
we obtain a function $a : \tilde\G = \Gr(\hh\G)\rightarrow\mathbb C,
\hh u\mapsto \ip{\hh u}{\hh \omega}$.  Clearly $a$ is continuous; the
content of the theorem is that $a\in A(\tilde\G)$, and that every member
of $A(\tilde\G)$ arises in this way.

The following improves \cite[Proposition~5.17]{kn}, in that we do not need to
assume that $\G$ is discrete.  We refer to \cite{bt} for the notion of a
\emph{coamenable} quantum group.

\begin{theorem}
Let $\hh\G$ be coamenable.  Then $\tilde\G$ is amenable.
\end{theorem}
\begin{proof}
That $\hh\G$ is coamenable is equivalent to $L^1(\hh\G)$ having a bounded
approximate identity.  As $A(\tilde\G)$ is a quotient of the Banach algebra
$L^1(\hh\G)$, it follows that $A(\tilde\G)$ also has a bounded approximate
identity, and so $\tilde\G$ is amenable, as claimed.
\end{proof}

Let us finish by observing the following corollary of the universal property
of $\tilde\G$, which shows that all ``classical'' closed quantum subgroups
are Vaes closed.

\begin{corollary}
Let $H$ be a locally compact group, identified as a (Woronowicz) closed
subgroup of $\G$.  Then $H$ is Vaes closed.
\end{corollary}
\begin{proof}
By Theorem~\ref{thm:uni_prop}, the morphism $H\rightarrow\G$ factors through
a homomorphism $H\rightarrow\tilde\G$.  To be precise, let the morphism
$H\rightarrow \G$ be given by a Hopf
$*$-homomorphism $\phi:C_0^u(\G) \rightarrow C_0(H)$, which is surjective by
assumption.  Let $\theta:C_0^u(\G) \rightarrow C_0(\tilde\G)$ be the surjective
Hopf $*$-homomorphism from Theorem~\ref{prop:gelfand}, so by
Theorem~\ref{thm:uni_prop} there is a Hopf $*$-homomorphism $\psi:C_0(\tilde\G)
\rightarrow M(C_0(H))$ with $\psi\circ\theta = \phi$.  As $\theta$ is surjective,
it follows that $\psi$ maps into and onto $C_0(H)$.  Thus $\psi$ identifies
$H$ as a closed subgroup of $\tilde\G$.

Equivalently, as explained after
Theorem~\ref{thm:uni_prop}, we have on the dual side that $C^*(H)
\rightarrow M(C_0^u(\hh\G))$ factors through $C^*(H)\rightarrow C^*(\tilde\G)$.
By Herz restriction, this map drops to an injective normal $*$-homomorphism
$VN(H) \rightarrow VN(\tilde\G)$,
and by the result above, we have $VN(\tilde\G) \rightarrow L^\infty(\hh\G)$.
The composition gives the required injective normal $*$-homomorphism
$VN(H)\rightarrow L^\infty(\hh\G)$.
\end{proof}

\vspace{5ex}

\noindent\emph{Author's Address:}
\parbox[t]{3in}{School of Mathematics\\
University of Leeds\\
Leeds\\
LS2 9JT\\
United Kingdom}

\bigskip\noindent\emph{Email:} \texttt{matt.daws@cantab.net}

\end{document}